\documentclass[12pt,reqno]{amsart}

\usepackage[T1]{fontenc}
\usepackage[utf8]{inputenc}
\usepackage[english]{babel}
\usepackage{amsthm,amssymb,amsmath}
\usepackage{hyperref}
\usepackage[top=1in, bottom=1in, left=1in, right=1in]{geometry}

\newtheorem{Theorem}{Theorem}[section]
\newtheorem{Lemma}[Theorem]{Lemma}
\newtheorem{Corollary}[Theorem]{Corollary}

\newtheorem{athm}{Theorem}

\theoremstyle{Definition}
\newtheorem{Definition}[Theorem]{Definition}

\newtheorem{Example}{Example}

\theoremstyle{Remark}

\DeclareMathOperator{\Aut}{Aut}

\title{A Schur--Zassenhaus Theorem for Finite Skew Braces}

\author[Marco Damele]{Marco Damele}
\thanks{Dipartimento di Matematica, Universit\`a di Cagliari, Via Ospedale 72, 09124 Cagliari, Italy; \texttt{marco.damele@unica.it}; ORCID 0009-0008-3088-5766}

\date{}

\begin{document}

\maketitle
\vspace{-1.8em}

\begin{abstract}
We prove a Schur--Zassenhaus theorem for finite skew braces. More precisely,
if \(B\) is a finite skew brace and \(I\) is an ideal of \(B\) such that
\(|I|\) and \(|B/I|\) are coprime, then \(I\) admits a complement in \(B\). As an application, we show that finite supersoluble skew braces are
Lagrangian, thereby recovering, for trivial skew braces, the classical
inverse Lagrange property for finite supersoluble groups.
\end{abstract}

\medskip


\medskip


\section{Introduction}

Skew braces were introduced by Guarnieri and Vendramin \cite{GV} as a
non-abelian generalisation of braces, originally introduced by Rump
\cite{Rump} in connection with set-theoretic solutions of the Yang--Baxter
equation. They are also closely related to regular subgroups of holomorphs
and have applications to the study of Hopf--Galois structures.

A skew brace is a triple \((B,+,\cdot)\), where \(B\) is a 
set and both \((B,+)\) and \((B,\cdot)\) are groups satisfying the
compatibility condition
\[
a\cdot (b+c)=a\cdot b-a+a\cdot c
\]
for all \(a,b,c\in B\), where \(-a\) denotes the inverse of \(a\) with
respect to the operation \(+\).  This compatibility condition can be translated by saying that the multiplicative group \((B,\cdot)\) acts
on its additive group \((B,+)\) through the lambda map
\(\lambda:(B,\cdot)\to \Aut(B,+)\), defined by
\[
\lambda_a(b)=-a+a\cdot b
\]
for all \(a,b\in B\). This map is a group homomorphism, as shown in
\cite{GV}. \emph{From now on, all skew braces considered in this paper are assumed to be
finite.}  Every finite group \(G\) can be regarded as a skew brace by taking the two operations to coincide, namely by considering the trivial skew brace \((G,\cdot,\cdot)\). It is therefore natural to ask which classical results from finite group
theory admit analogues in the wider setting of finite skew braces. Among the
most fundamental results in finite group theory are Sylow's theorems and
Hall's theorem, which describe the existence and behaviour of subgroups of
prescribed prime-power or coprime order. The natural analogues of subgroups and normal subgroups in the setting of
skew braces are sub-skew braces and ideals, respectively. A sub-skew brace of \(B\) is
a subset which is a subgroup of both \((B,+)\) and \((B,\cdot)\). An ideal
is a sub-skew brace \(I\) which is normal in both groups and is invariant under
the maps \(\lambda_a\), for every \(a\in B\); equivalently, \(B/I\) is again
a skew brace. Consequently, the notions of Sylow sub-skew brace and Hall
\(\pi\)-sub-skew brace are the natural extensions of the corresponding notions
for finite groups: a Sylow \(p\)-sub-skew brace is a sub-skew brace whose order is the
largest power of \(p\) dividing \(|B|\), while a Hall \(\pi\)-sub-skew brace is a
sub-skew brace whose order is a \(\pi\)-number and whose index is a
\(\pi'\)-number.

The existence problem for Sylow sub-skew braces was first approached in
\cite{CDDMFT}, where Sylow and Hall type results were obtained for special
classes of finite skew braces. The existence of Sylow sub-skew braces was
later proved in full generality in \cite{Truman}, together with an analogue
of Cauchy's theorem. Under the additional assumption that both the additive
and the multiplicative groups are soluble, \cite{Truman} also proves an
analogue of the existence part of Hall's theorem. Shortly after,
\cite{BBPPCtrifactorised} showed the power of the trifactorised group
associated with a skew brace: the Sylow and Hall type results of
\cite{Truman} can be recovered from the corresponding Hall theory for finite
trifactorised groups. For the class of finite left nilpotent skew braces, a
stronger containment property was recently obtained in \cite{EGGK}: not only
do Sylow \(p\)-sub-skew braces exist, but every \(p\)-sub-skew brace is
contained in a Sylow \(p\)-sub-skew brace.

The aim of this note is to prove an analogue for finite skew braces of the
classical theorem of Schur and Zassenhaus. Recall that, in finite group
theory, this theorem asserts that a normal subgroup \(N\) of a finite group
\(G\) admits a complement whenever \(|N|\) and \(|G/N|\) are coprime
(see \cite[Theorem 6.2.1]{KurzweilStellmacher}). In this note, by a
complement of an ideal \(I\) in a skew brace \(B\) we mean a sub-skew brace
\(H\) of \(B\) such that \(I\cap H=\{0\}\), \((B,+)=I+H\), and
\((B,\cdot)=IH\). Complement problems for skew braces have already been
studied in \cite{RatheeYadav} by means of extensions and second cohomology.
The splitting results obtained there require several additional hypotheses.
Motivated by the trifactorised-group approach of
\cite{BBPPCtrifactorised}, we show that no further assumptions are needed
for coprime ideals of finite skew braces.

\begin{athm}\label{MainTheorem}
Let $B$ be a finite skew brace and let $I$ be an ideal of $B$.
If $|I|$ and $|B/I|$ are coprime, then $I$ admits a complement in $B$. 
\end{athm}

The proof is inspired by the trifactorised-group approach of \cite{BBPPCtrifactorised}. We combine the classical Schur--Zassenhaus theorem with the Hall theorem for finite trifactorised groups. The proof uses the Feit--Thompson theorem to ensure that one of the two coprime factors is soluble, thereby allowing the conjugacy statement in the Schur--Zassenhaus theorem to be applied.

We stress that Theorem \ref{MainTheorem} concerns the existence of complements. A corresponding analogue of the conjugacy statement in the classical Schur–Zassenhaus theorem appears to be more delicate in the skew brace setting, since a complement must be compatible with both the additive and the multiplicative group structures. We do not address this question here.

After the proof of Theorem \ref{MainTheorem}, we also give an example showing that
one cannot expect a full analogue of the containment part of Sylow theory:
there exists a finite skew brace with a \(2\)-sub-skew brace which is not contained
in any Sylow \(2\)-sub-skew brace.

As an application, we use Theorem \ref{MainTheorem} to prove that every finite supersoluble
skew brace is Lagrangian, in the sense that it contains a sub-skew brace of
order \(d\) for every divisor \(d\) of its order. Since groups can be regarded
as trivial skew braces, this recovers the classical result that every finite
supersolvable group is Lagrangian.

\section{Proof of Theorem \ref{MainTheorem}}

Before proving the main theorem, we recall some standard definitions from
finite group theory. For a positive integer \(n\), we denote by \(\pi(n)\) the set of prime
divisors of \(n\). Let \(G\) be a finite group and let \(\pi\) be a set of
primes. A subgroup \(H\leq G\) is called a \(\pi\)-subgroup if every prime
divisor of \(|H|\) belongs to \(\pi\). A subgroup \(H\leq G\) is called a Hall \(\pi\)-subgroup of \(G\) if
\(\pi(|H|)\subseteq \pi\) and \(\pi([G:H])\cap \pi=\varnothing\).

\begin{Definition} \rm
  Let \(G\) be a finite group and let \(\pi\) be a set of primes. We say that
\(G\) satisfies the \(D_\pi\)-property, or that \(G\) is a \(D_\pi\)-group,
if the following two conditions hold:
\begin{enumerate}
    \item \(G\) has Hall \(\pi\)-subgroups;
    \item every \(\pi\)-subgroup of \(G\) is contained in a Hall
    \(\pi\)-subgroup of \(G\), and any two Hall \(\pi\)-subgroups of \(G\)
    are conjugate in \(G\).
\end{enumerate}  
\end{Definition}

Observe that the following consequence of the Schur--Zassenhaus theorem
will be used several times.

\begin{Lemma}\label{Remark}
Let \(X\) be a finite group and let \(M\unlhd X\) be a normal Hall
\(\pi'\)-subgroup. Then \(X\) satisfies the \(D_\pi\)-property.
\end{Lemma}

\begin{proof}
By the Feit--Thompson theorem (see the main theorem of \cite{FeitThompson1963}), at least one of \(M\) and \(X/M\) is
soluble. Indeed, since \(M\) is a Hall \(\pi'\)-subgroup of \(X\), the
integers \(|M|\) and \(|X/M|\) are coprime. Hence at least one of them is
odd. If \(|M|\) is odd, then \(M\) is soluble; if \(|X/M|\) is odd, then
\(X/M\) is soluble. By the Schur--Zassenhaus theorem \cite[Theorem~6.2.1]{KurzweilStellmacher},
\(M\) has a complement \(L\) in \(X\). Thus \(X=ML\) and \(M\cap L=1\), so
\(L\) is a Hall \(\pi\)-subgroup of \(X\). Moreover, since either \(M\) or
\(X/M\) is soluble, all complements of \(M\) in \(X\) are conjugate (again by \cite[Theorem~6.2.1]{KurzweilStellmacher}). Hence
all Hall \(\pi\)-subgroups of \(X\) are conjugate. 

It remains to show that every \(\pi\)-subgroup of \(X\) is contained in a
Hall \(\pi\)-subgroup of \(X\). Let \(P\leq X\) be a \(\pi\)-subgroup and
set \(Y=PM\). Since \(M\unlhd X\), the set \(Y\) is a subgroup of \(X\), and
\(M\unlhd Y\). Moreover, \(P\cap M=1\), because \(P\) is a \(\pi\)-group and
\(M\) is a \(\pi'\)-group. Hence \(P\) is a complement of \(M\) in \(Y\).

We now construct another complement of \(M\) in \(Y\) contained in \(L\). Set
\(L_0=L\cap Y\). We claim that \(Y=ML_0\). Indeed, let \(y\in Y\). Since
\(X=ML\), there exist \(m\in M\) and \(l\in L\) such that \(y=ml\). Since
\(y\in Y\) and \(m\in M\leq Y\), it follows that \(l=m^{-1}y\in Y\). Hence
\(l\in L\cap Y=L_0\), and so \(y\in ML_0\). Thus \(Y\subseteq ML_0\), while
the reverse inclusion is clear because \(M\leq Y\) and \(L_0\leq Y\).
Therefore \(Y=ML_0\). Moreover, \(M\cap L_0=1\), since \(L_0\leq L\) and
\(M\cap L=1\). Hence \(L_0\) is a complement of \(M\) in \(Y\).

Again, either \(M\) or \(Y/M\) is soluble: indeed, if \(M\) is soluble there
is nothing to prove, while if \(X/M\) is soluble then also \(Y/M\leq X/M\)
is soluble. Therefore the conjugacy part of the Schur--Zassenhaus theorem
applies inside \(Y\). Since \(P\) and \(L_0\) are two complements of \(M\) in
\(Y\), there exists \(y\in Y\) such that \(P^y=L_0\). Hence
\(P\leq L^{y^{-1}}\). Since \(L^{y^{-1}}\) is a Hall \(\pi\)-subgroup of
\(X\), the subgroup \(P\) is contained in a Hall \(\pi\)-subgroup of \(X\).

Consequently, every \(\pi\)-subgroup of \(X\) is contained in a Hall
\(\pi\)-subgroup of \(X\), and all Hall \(\pi\)-subgroups of \(X\) are
conjugate. Thus \(X\) satisfies the \(D_\pi\)-property.
\end{proof}

A group \(G\) is said to be \emph{trifactorised } if there exist
subgroups \(K,C,D\leq G\) such that
\(G=KC=KD=DC\).

\begin{Theorem}{\cite[theorem 2]{BBPPCtrifactorised}} \label{Trifactorized}
Let \(\pi\) be a set of primes, and let $G=KC=KD=DC$
be a trifactorised group, where \(K\unlhd G\). Assume that \(G\), \(C\), and
\(D\) satisfy the \(D_\pi\)-property. Then there exist Hall \(\pi\)-subgroups $K_\pi\leq K, C_\pi\leq C, D_\pi\leq D$
such that
\[
G_\pi=K_\pi C_\pi=K_\pi D_\pi=D_\pi C_\pi
\]
is a Hall \(\pi\)-subgroup of \(G\).
\end{Theorem}

We are now ready to prove Theorem \ref{MainTheorem}.

\begin{proof}[Proof of Theorem \ref{MainTheorem}]
Let \(\lambda:(B,\cdot)\to \Aut(B,+)\) be the \(\lambda\)-map of \(B\). Consider the semidirect product associated to this map, namely the group \(G=(B,+)\rtimes_{\lambda}(B,\cdot)\). Let \(K=\{(h,1):h\in B\}\), \(C=\{(0,h):h\in B\}\), and
\(D=\{(h,h):h\in B\}\). Then, as recalled in \cite{BBEPPC}, the group $G$ satisfies
\[
G=KC=KD=DC,
\]
with \(K\unlhd G\). Thus \(G\) is trifactorised.
Set \(a=|I|\) and \(b=|B/I|\), and let \(\pi=\pi(b)\). We claim that the groups \(G\), \(C\), and \(D\) satisfy the
\(D_\pi\)-property.

\begin{enumerate}
 
   \item For \(G\), consider the subgroup
\(N=(I,+)\rtimes_{\lambda}(I,\cdot)\) of
\(G=(B,+)\rtimes_{\lambda}(B,\cdot)\). Since \(I\) is an ideal of \(B\),
it is normal in both \((B,+)\) and \((B,\cdot)\), and it is invariant
under the maps \(\lambda_a\), for all \(a\in B\). Therefore
\(N\unlhd G\). Moreover, \(|N|=|I|^2=a^2\) and
\(|G/N|=|B/I|^2=b^2\). Hence \(N\) is a normal Hall \(\pi'\)-subgroup
of \(G\). Thus \(G\) satisfies the \(D_\pi\)-property by
Lemma~\ref{Remark}.

    \item For \(C=(B,\cdot)\), since \(I\) is an ideal of \(B\), the
    subgroup \(C_I=(I,\cdot)\) is normal in \(C\). Moreover,
    \(|C_I|=|I|=a\) and \(|C/C_I|=|B/I|=b\). Hence \(C_I\) is a normal Hall
    \(\pi'\)-subgroup of \(C\). Thus \(C\) satisfies
    the \(D_\pi\)-property by Lemma~\ref{Remark}. 
  
    \item For \(D\), let \(D_I\) be the diagonal subgroup associated to
    \(I\). Since \(I\) is an ideal of \(B\), we have \(D_I\unlhd D\).
    Moreover, \(D_I\) is a normal Hall \(\pi'\)-subgroup of \(D\), with
    \(|D_I|=a\) and \(|D/D_I|=b\). Thus \(D\)
    satisfies the \(D_\pi\)-property by Lemma~\ref{Remark}. 
\end{enumerate}

Therefore the trifactorised group \(G=KC=KD=DC\) satisfies the hypotheses
of Theorem~\ref{Trifactorized}. Hence there exist Hall
\(\pi\)-subgroups \(K_\pi\leq K\), \(C_\pi\leq C\), and \(D_\pi\leq D\)
such that $G_\pi=K_\pi C_\pi=K_\pi D_\pi=D_\pi C_\pi$
is a Hall \(\pi\)-subgroup of \(G\). We now verify explicitly that, under the natural identification
\(K\simeq (B,+)\), the subgroup \(K_\pi\) corresponds to a sub-skew brace of
\(B\). This is not a consequence of \(K_\pi\leq K\) alone; it uses the full
trifactorisation
\(G_\pi=K_\pi C_\pi=K_\pi D_\pi=D_\pi C_\pi\). Define
\begin{align*}
H_K &= \{h\in B : (h,1)\in K_\pi\},\\
H_C &= \{h\in B : (0,h)\in C_\pi\},\\
H_D &= \{h\in B : (h,h)\in D_\pi\}.
\end{align*}
We claim that the three subsets \(H_K\), \(H_C\), and \(H_D\) of \(B\)
coincide.

Let \(h\in H_D\). By definition, \((h,h)\in D_\pi\). Since
\(D_\pi\leq G_\pi\) and \(G_\pi=K_\pi C_\pi\), we have
\((h,h)\in K_\pi C_\pi\). Hence there exist elements \(k\in B\) and
\(c\in B\) such that \((k,1)\in K_\pi\), \((0,c)\in C_\pi\), and
\[
(h,h)=(k,1)(0,c)=(k,c).
\]
Thus \(k=h\) and \(c=h\). Therefore
\((h,1)\in K_\pi\) and \((0,h)\in C_\pi\), that is, \(h\in H_K\cap H_C\).
This proves that \(H_D\subseteq H_K\cap H_C\).
We now compare the orders. The map \(B\to K\), \(h\mapsto (h,1)\), is a
bijection, and by definition it restricts to a bijection
\(H_K\to K_\pi\). Hence \(|H_K|=|K_\pi|\). Similarly, the maps
\(h\mapsto (0,h)\) and \(h\mapsto (h,h)\) give bijections
\(H_C\to C_\pi\) and \(H_D\to D_\pi\), respectively. Therefore
\[
|H_K|=|K_\pi|,\qquad |H_C|=|C_\pi|,\qquad |H_D|=|D_\pi|.
\]
Since \(K_\pi\), \(C_\pi\), and \(D_\pi\) are Hall \(\pi\)-subgroups of
\(K\), \(C\), and \(D\), respectively, and since
\(|K|=|C|=|D|=|B|=ab\), we have
\[
|K_\pi|=|C_\pi|=|D_\pi|=b.
\]
Consequently,
\[
|H_K|=|H_C|=|H_D|=b.
\]
Since \(H_D\subseteq H_K\) we
obtain \(H_D=H_K\). Similarly, from \(H_D\subseteq H_C\) and
\(|H_D|=|H_C|\), we obtain \(H_D=H_C\). Hence
\[
H_K=H_C=H_D.
\] Let \(H\) denote this common subset of \(B\). Since \(H=H_K\), the subset \(H\) is a subgroup of \((B,+)\). Since \(H=H_C\), the subset \(H\) is also a
subgroup of \((B,\cdot)\). Hence
\(H\) is a sub-skew brace of \(B\) of order $b$.
Thus \(B\) contains a
sub-skew brace of order \(b=|B/I|\). 
We now show that \(H\) is a complement of \(I\). Since
\(|H|=|B/I|\) and \(|I|\) is coprime to \(|B/I|\), the subgroups
\(I\) and \(H\) have coprime orders. Hence \(I\cap H=\{0\}\).
Moreover, since \(I\trianglelefteq (B,+)\), the set \(I+H\) is a
subgroup of \((B,+)\), and
\(|I+H|=|I||H|/|I\cap H|=|B|\). Thus \((B,+)=I+H\).
Similarly, since \(I\trianglelefteq (B,\cdot)\), the set \(IH\) is a
subgroup of \((B,\cdot)\), and
\(|IH|=|I||H|/|I\cap H|=|B|\). Thus \((B,\cdot)=IH\).
Therefore \(H\) is a complement of \(I\) in  \(B\).

\end{proof}

\begin{Example} \rm
We show that, in general, a \(\pi\)-sub-skew brace of a finite skew brace need
not be contained in a Hall \(\pi\)-sub-skew brace. Let \(B=S_3\times C_2\).
We write the elements of \(B\) as pairs \((g,i)\), where \(g\in S_3\) and
\(i\in C_2=\{0,1\}\). Fix a transposition \(s\in S_3\), and denote by
\(\cdot\) the direct product operation on \(B\). Define a second operation
\(\circ\) on \(B\) by
\begin{equation*}
(g,i)\circ(h,j)=(g s^i h s^{-i},i+j).
\end{equation*}
Observe that \((B,\circ)\) is the semidirect product
\(S_3\rtimes C_2\), where the non-trivial element of \(C_2\) acts on 
\(S_3\) by conjugation by \(s\). We first explain why \((B,\cdot,\circ)\) is a skew brace. Let
\(\theta\in \Aut(B,\cdot)\) be defined by
\(\theta(h,j)=(shs^{-1},j)\). Then
\(\theta^i(h,j)=(s^i h s^{-i},j)\), and therefore
\begin{equation*}
\lambda_{(g,i)}(h,j)
:=(g,i)^{-1_{\cdot}}\cdot \bigl((g,i)\circ(h,j)\bigr)
=(s^i h s^{-i},j)
=\theta^i(h,j).
\end{equation*}
Thus \(\lambda_{(g,i)}=\theta^i\). Since the second coordinate of
\((g,i)\circ(h,j)\) is \(i+j\), we get
\begin{equation*}
\lambda_{(g,i)\circ(h,j)}
=\theta^{i+j}
=\theta^i\theta^j
=\lambda_{(g,i)}\lambda_{(h,j)}.
\end{equation*}
Hence \(\lambda:(B,\circ)\to \Aut(B,\cdot)\) is a group homomorphism.
By the standard criterion for skew braces, see
\cite[Proposition~1.9]{GV}, the triple \((B,\cdot,\circ)\) is a skew
brace.
We consider the prime \(2\). Since \(|B|=12=2^2\cdot 3\), a Hall
\(\{2\}\)-sub-skew brace, equivalently a Sylow \(2\)-sub-skew brace, has
order \(4\). The Sylow \(2\)-subgroups of \((B,\cdot)\) are
\begin{equation*}
P_t=\{(1,0),(1,1),(t,0),(t,1)\},
\end{equation*}
where \(t\) runs through the transpositions of \(S_3\). We claim that
\(P_t\) is a sub-skew brace if and only if \(t=s\). Indeed, suppose that
\(P_t\) is closed under \(\circ\). Since \((1,1),(t,0)\in P_t\), we must
have
\begin{equation*}
(1,1)\circ(t,0)=(sts^{-1},1)\in P_t.
\end{equation*}
This forces \(sts^{-1}=t\). Among the transpositions of \(S_3\), this holds
if and only if \(t=s\), because the centralizer of a transposition in
\(S_3\) is generated by that transposition. Conversely, if \((s^a,i),(s^b,j)\in P_s\), with \(a,b,i,j\in\{0,1\}\), then
\[
(s^a,i)\circ(s^b,j)=(s^a s^i s^b s^{-i},i+j)=(s^{a+b},i+j)\in P_s,
\]
so \(P_s\) is closed under \(\circ\). Therefore \(P_s\) is the unique Hall
\(\{2\}\)-sub-skew brace of \(B\).
Now let \(t\neq s\) be another transposition and set
\begin{equation*}
A=\{(1,0),(t,0)\}.
\end{equation*}
Then \(A\) is a \(2\)-sub-skew brace of \(B\). Indeed, \(A\) is a subgroup
of \((B,\cdot)\), and on \(A\) the operation \(\circ\) coincides with
\(\cdot\), because all elements of \(A\) have second coordinate \(0\). In
particular, \((t,0)\circ(t,0)=(t^2,0)=(1,0)\). Thus \(A\) is a
sub-skew brace of order \(2\). However, \(A\) is not contained in the unique
Hall \(\{2\}\)-sub-skew brace \(P_s\), since \((t,0)\notin P_s\). Hence
\(B\) contains a \(\{2\}\)-sub-skew brace which is not contained in any Hall
\(\{2\}\)-sub-skew brace.
\end{Example}

\section{Supersoluble skew braces and the Lagrangian property}

In this section we record a consequence of Theorem \ref{MainTheorem} for supersoluble skew
braces. We use the notion of supersolubility for skew braces as in
\cite{BBEFPT}. Thus a finite skew brace \(B\) is called supersoluble if it
has a series of ideals
\[
0=B_0\leq B_1\leq \cdots \leq B_n=B
\]
such that \(|B_i/B_{i-1}|\) is a prime number for every \(i=1,\ldots,n\).

We first recall the following elementary fact.

\begin{Lemma}
Let \(B\) be a finite supersoluble skew brace. Then every minimal non-zero
ideal of \(B\) has prime order.
\end{Lemma}

\begin{proof}
Let
\[
0=B_0\leq B_1\leq \cdots \leq B_n=B
\]
be a supersoluble series of ideals of \(B\), with \(|B_i/B_{i-1}|\) prime for
every \(i\). Let \(I\) be a minimal non-zero ideal of \(B\). Choose \(j\) minimal such that \(I\cap B_j\neq 0\). Then
\(I\cap B_{j-1}=0\). Since \(I\) and \(B_j\) are ideals of \(B\), also
\(I\cap B_j\) is an ideal of \(B\). Moreover \(I\cap B_j\) embeds into
\(B_j/B_{j-1}\). Indeed,
\[
(I\cap B_j)\cap B_{j-1}=I\cap B_{j-1}=0.
\]
Since \(B_j/B_{j-1}\) has prime order and \(I\cap B_j\neq 0\), it follows that
\[
|I\cap B_j|=|B_j/B_{j-1}|
\]
is prime. Finally, \(I\cap B_j\) is a non-zero ideal of \(B\) contained in
\(I\). By the minimality of \(I\), we have \(I=I\cap B_j\). Hence \(|I|\) is
prime.
\end{proof}

We now introduce the analogue of the inverse Lagrange property.

\begin{Definition}
A finite skew brace \(B\) is said to be \emph{Lagrangian} if, for every
divisor \(d\) of \(|B|\), there exists a sub-skew brace of \(B\) of order
\(d\).
\end{Definition}

The proof of the following corollary is the standard group-theoretic argument, with Theorem \ref{MainTheorem} replacing the classical Schur--Zassenhaus theorem.

\begin{Corollary}
Every finite supersoluble skew brace is Lagrangian.
\end{Corollary}

\begin{proof}
We argue by induction on \(|B|\). If \(|B|=1\), there is nothing to prove.
Assume that \(B\neq 0\), and let \(d\) be a divisor of \(|B|\). Since \(B\) is supersoluble, by the previous lemma \(B\) has a minimal non-zero
ideal \(I\) of prime order, say \(|I|=p\). Moreover \(B/I\) is again
supersoluble. Suppose first that \(p\mid d\). Then \(d/p\) divides \(|B/I|\). By induction,
there exists a sub-skew brace \(\overline H\) of \(B/I\) of order \(d/p\).
Let \(H\) be the inverse image of \(\overline H\) under the natural projection
\(B\to B/I\). Then \(H\) is a sub-skew brace of \(B\), it contains \(I\), and
\[
|H|=|I|\,|\overline H|=p\cdot d/p=d.
\]
Suppose now that \(p\nmid d\). Then \(d\) divides \(|B/I|\). By induction,
there exists a sub-skew brace \(\overline K\) of \(B/I\) of order \(d\). Let
\(K\) be the inverse image of \(\overline K\) in \(B\). Then \(K\) is a
sub-skew brace of \(B\), it contains \(I\), and
\[
|K|=|I|\,|\overline K|=pd.
\]
Furthermore, \(I\) is an ideal of \(K\). Since
\[
(|I|,|K/I|)=(p,d)=1,
\]
Theorem A applied to the finite skew brace \(K\) and to its ideal \(I\) gives
a sub-skew brace \(L\) of \(K\) such that
\[
K=I+L=I\circ L,\qquad I\cap L=0.
\]
In particular,
\[
|L|=|K/I|=d.
\]
Thus in both cases \(B\) contains a sub-skew brace of order \(d\). Therefore
\(B\) is Lagrangian.
\end{proof}

\end{document}